\documentclass[12pt]{article}

\usepackage{amstext}
\usepackage{amsthm}
\usepackage{amsmath}
\usepackage{amssymb}
\usepackage{latexsym}
\usepackage{amsfonts}
\usepackage{color}
\usepackage{graphicx}

\usepackage[pagebackref,hypertexnames=false, colorlinks, citecolor=black, linkcolor=black, urlcolor=red]{hyperref}

\definecolor{dark_purple}{rgb}{0.4, 0.0, 0.4}
\definecolor{dark_green}{rgb}{0.0, 0.7, 0.0}

\renewcommand\d{\delta}
\newcommand\G{{\mathcal G}}
\newcommand\DDD{{\mathcal D}}
\newcommand\BB{{\mathcal B}}

\usepackage{latexsym}
\usepackage{amssymb}
\usepackage{euscript}

\let\cal=\mathcal      

\def\mcc{M\raise.5ex\hbox{c}C}
\def\mccarthy{M\raise.5ex\hbox{c}Carthy}

\def\ie{{\it i.e. }}

\def\M{{\cal M}}

\def\vare{\varepsilon}

\let\i=\infty

\def\la{\langle}
\def\ra{\rangle}
\def\={\ = \ }

\def\A{{\cal A}}
\def\BB{{\cal B}}

\def\F{{\cal F}}

\def\C{\mathbb C}

\def\T{\mathbb T}
\def\D{\mathbb D}
\def\B{\mathbb B}

\def\be{\setcounter{equation}{\value{theorem}} \begin{equation}}
\def\ee{\end{equation} \addtocounter{theorem}{1}}
\def\beq{\begin{eqnarray*}}
\def\eeq{\end{eqnarray*}}
\def\se{\setcounter{equation}{\value{theorem}}} 
\def\att{\addtocounter{theorem}{1}}
\def\vs{\vskip 5pt}
\def\bs{\vskip 12pt}

\def\bp{{\sc Proof: }}
\def\ep{{}{\hfill $\Box$} \vskip 5pt \par}

\def\oec{{}{\hspace*{\fill} $\lhd$} \vskip 5pt \par}
\def\bl{\begin{lemma}}
\def\el{\end{lemma}}
\def\bt{\begin{theorem}}
\def\et{\end{theorem}}
\def\bprop{\begin{prop}}
\def\eprop{\end{prop}}
\def\bd{\begin{definition}}
\def\ed{\end{definition}}
\def\br{\begin{remark}}
\def\er{\end{remark}}
\def\bexer{\begin{exercise}}
\def\eexer{\end{exercise}}

\newtheorem{theorem}{Theorem}[section]
\newtheorem{proposition}[theorem]{Proposition}
\newtheorem{lemma}[theorem]{Lemma}
\newtheorem{corollary}[theorem]{Corollary}

\newtheorem{definition}[theorem]{Definition}

\newtheorem{defin}[theorem]{Definition}

\renewcommand\Re{\mathrm{Re\, }}

\renewcommand\O{\Omega}
\newcommand\hio{H^\infty(\Omega)}
\newcommand\bak{{\B}_k}
\newcommand\bad{{\B}_d}
\newcommand\stretchit{r}
\newcommand\lam{\lambda}
\newcommand\atmu{{A}^2(\mu)}

\numberwithin{equation}{section}
\title{Norm preserving extensions of bounded holomorphic functions}
\author{\L{}ukasz Kosi\'nski
\thanks{Partially supported by the NCN grant UMO-2014/15/D/ST1/01972}
\and
John E. M\raise.5ex\hbox{c}Carthy
\thanks{Partially supported by National Science Foundation Grant  
DMS 156243}
}

\begin{document}

\bibliographystyle{plain}
\maketitle

\section{Introduction}

\subsection{Statement of Results}

Let $\O$ be a  set in $\C^d$, and $V$ be a subset of $\O$, with no extra structure assumed. 
A function $f : V \to \C$ is
said to be {\em holomorphic} if, for every point $\lambda \in V$, there exists $\vare >0$ and a holomorphic function
$F$ defined on the ball $B(\lambda, \vare)$ in $\C^d$ such that $F$ agrees with $f$ on 
$V \cap B(\lambda, \vare)$. Let  $\A$ be an algebra of bounded holomorphic functions on $V$,
equipped with the sup-norm.

 \bd
  We say $V$ has the {\em $\A$
extension property} if, for every $f$ in $\A$, there exists $F \in \hio$, such that
$F |_V = f$ and $\| F \|_\O = \| f \|$.
If  $\O$ is a bounded set, and 
$\A$ is the algebra of polynomials, we shall say $V$ has the extension property.
\ed

If $\O$ is pseudo-convex, and $V$ is an analytic subvariety of $\O$, it is a deep theorem of H. Cartan that every holomorphic function on $V$ extends to a holomorphic function on $\O$ \cite{car51}.
This is not true in general for extensions that preserve the $H^\i$-norm.
There is one easy way to have a norm-preserving extension.  We say $V$ is a {\em retract } 
of $\O$ if there is a holomorphic map $r: \O \to \O$
such that the range of $r$ is $V$ and $r |_V$ is the identity. If $V$ is a retract, then
$f \circ r$ will always be a norm-preserving extension of $f$ to $\O$.

In \cite{agmc_vn}, 
it was shown that if $\Omega$ is the bidisk, that is basically the only way that sets can have the extension property.

\bt
\label{thma1} \cite{agmc_vn}
 Let $V$ be a relatively polynomially convex subset of $\D^2$. Then $V$ has the extension property if and only if $V$ is a retract of $\D^2$.
 \et
 
We say that a set $V$ contained in a domain $\O$ is {\em  relatively polynomially convex} if the intersection of the
polynomial hull $\hat V$  with $\Omega$ is $V$. If $\hat V \cap \O$ has the extension property, so does $V$,
so the assumption of relative polynomial convexity is a natural one.


 In \cite{aly16}, J. Agler, Z. Lykova and N. Young proved 
  that, for the symmetrized bidisk, that is the set
 \be
 \label{eqa01}
 G \= \{ (z+w, zw) : z,w \in \D \} ,
 \ee
 not all sets with the extension property are retracts. They proved:
 \bt
 \label{thma2} \cite{aly16}
 The set $V$ is an algebraic subset of $G$ having the $H^\i(V)$ extension property if and only if either
 $V$ is a retract of $G$, or $V = {\mathcal R} \cup {\mathcal D}_\beta$, where
 ${\mathcal R} = \{(2z,z^2) : z \in \D\}$ and ${\mathcal D}_\beta = \{ (\beta + \bar \beta z, z):
 z \in \D \}$, and $\beta \in \D$.
 \et
 
 It is the purpose of this note to study the extension property for domains other than the bidisk
 and symmetrized bidisk.
 Our first main result is  for strictly convex bounded sets in $\C^2$:
 
 \bt
 \label{thma3}
 Let $\O$ be a strictly convex bounded subset of $\C^2$, and assume that $V \subseteq \O$ 
 is relatively polynomially convex. Then $V$ has the extension property if and only if $V$ is a retract of
 $\O$.
 \et

We prove Theorem~\ref{thma3} in Section \ref{secb}. In Section \ref{secc} we prove a similar result for balls in any dimension.
 In Section \ref{secd} we consider strongly linearly convex domains.
 We shall give a definition of strongly linearly convex in Section~\ref{secd};
roughly it says that the domain does not have second order contact with complex tangent planes.

   We prove:
 \bt
 \label{thma4}
 Let $\O \subseteq \C^d$ be a strongly linearly convex
 bounded
  domain with ${\mathcal C}^3$ boundary,
 and assume that $V \subseteq \O$ 
 is relatively polynomially convex, and not a singleton. If $V$ has the extension property, then $V$ is a totally geodesic
 complex submanifold of $\O$.
 \et

As a corollary of Theorem~\ref{thma4}, we conclude:
\begin{corollary}
\label{cora1}
 Let  $\O \subseteq \C^2$
 be a strongly linearly convex
 bounded
  domain with $C^3$ boundary,
 and assume that $V \subseteq \O$ 
 is relatively polynomially convex. If $V$ has the extension property, then $V$ is a
 retract.
 \end{corollary}
In \cite{pz12}, Pflug and Zwonek proved that the symmetrized bidisk is an increasing union of strongly linearly 
convex domains with smooth (even real analytic) boundaries.
So contrasting Theorem~\ref{thma2} with Corollary~\ref{cora1} shows that the extension property implying a retract is not stable under increasing limits.

\subsection{Motivation and history}

One reason to study sets with the extension property is provided by spectral sets.
Let $\O$ be an open set in $  \C^d$. If $T = (T_1, \dots, T_d)$  is a $d$-tuple 
 of commuting operators on a 
Hilbert space with spectrum in $\O$, then one can define $f(T)$ for any $f \in H^\i(\O)$ by the Taylor functional calculus \cite{tay70b}.
We say $\O$ is a {\em spectral set}  for $T$ 
 if the following analogue of von Neumann's inequality holds:
\be
\label{eqab1}
\| f(T) \| \ \leq \ \sup_{\lambda \in \O} |f(\lambda)|  \quad \forall f \in H^\i(\O).
\ee
If $V \subseteq \O$, we say that $T$ is {\em  subordinate to $V$}  if the spectrum of $T$ is in $V$ and
$f(T) = 0$ whenever $f \in H^\i(\O)$ and $f |_V = 0$.

\begin{defin}
Let $V \subseteq \O$, a domain in $\C^d$, and let $\A $ be an algebra of functions that are holomorphic in some neighborhood of $V$. We say $V$ is an $\A$ von Neumann set w.r.t. $\O$
if, whenever $T$ is a $d$-tuple of commuting operators on a Hilbert space that has $\O$ as a spectral set
and that is subordinate to $V$, then
\be
\label{eqab2}
\| f(T) \| \ \leq \ \sup_{\lambda \in V} |f(\lambda)|  \quad \forall f \in \A.
\ee
\end{defin}
Note the big difference between \eqref{eqab1} and \eqref{eqab2} is whether the norm of $f(T)$
is controlled by  just 
the values of $f$ on $V$, or all of the values on $\O$.

Let $H^\i(V)$ denote the algebra of all functions holomorphic in a neighborhood of $V$, equipped with the supremum norm. One of the main results of \cite{agmc_vn} is this:
\bt
\label{thmab1} Let $\Omega$ be the bidisk $\D^2$, and  $V$ be a subset. Let $\A$ be a sub-algebra
of $H^\i(V)$. Then $V$ is an $\A$ von Neumann set if and only if it has the $\A$ extension property.
\et
In \cite{aly16}, the same theorem is proved when $\O$ is the symmetrized bidisk \eqref{eqa01}.
In Section~\ref{sece}, we prove that the theorem holds for any bounded domain $\O$ and any algebra containing the polynomials.
\bt
\label{thmab2}
Let $\O$ be a bounded domain in $\C^d$, and let $V \subseteq \O$. 
 Let $\A$ be a sub-algebra
of $H^\i(V)$.
 Then $V$ is an $\A$ von Neumann set if and only if it has the $\A$ extension property.
\et

Another reason to study sets with the extension property is if one wishes to understand Nevanlinna-Pick interpolation.
Given a domain $\O$ and $N$ distinct points $\lambda_1, \dots, \lambda_N$ in $\O$, the Nevanlinna-Pick problem is 
to determine, for each given set $w_1, \dots, w_N$ of complex numbers, 
\[
\inf  \{ \| \phi \|_{H^\i(\O)}  \ : \ \phi(\lambda_i) = w_i, \ \forall \ 1 \leq i \leq N \},
\]
and to describe the minimal norm solutions.
This problem has been extensively studied in the disk \cite{bgr90}, where the minimal norm solution is always unique, but is more elusive in higher dimensions.
Tautologically there is some holomorphic subvariety $V$ on which all minimal norm solutions coincide, but sometimes one
can actually say something descriptive about $V$, as in \cite{agmc_dv, kosthree}.
If $V$ had the extension property, one could split the analysis into two pieces: finding the unique solution on $V$,
and then studying how it extends to $\O$.

The first result we know of norm-preserving extensions is due to W. Rudin \cite[Thm 7.5.5]{rud69},
who showed that if $V$ is an embedded polydisk in $\Omega$, and the extension operator from $H^\i(V)$ to
$H^\i(\O)$  is linear of norm one, then
$V$ must be a retract of $\O$. Theorem~\ref{thma1} from \cite{agmc_vn} characterized sets in $\D^2$ that have the extension property,
and Theorem~\ref{thma2} from \cite{aly16} did this for the symmetrized bidisk. Neither of these results assume that there is a linear extension operator.
In \cite{ghw08}, Guo, Huang and Wang proved 
\bt
Suppose $V$ is an algebraic subvariety of  $\D^3$ that has the $H^\i(V)$ extension property, and there is a linear extension operator from $H^\i(V)$ to
$H^\i(\D^3)$. Then $V$ is a retract of $\D^3$.

If $V$ is an $H^\infty$-convex subset of $\D^d$ for any $d$, it has the  $H^\i(V)$ extension property, and the extension operator is an algebra homomorphism, then $V$ is a retract of $\D^d$.
\et

\section{Strictly convex domains in $\C^2$}
\label{secb}

A convex set $\O$ in $\C^d$ is called {\em strictly convex} if for every boundary point $\lambda$,
there is a real hyperplane $P$ such that $P \cap \overline{\O} = \{ \lambda \}$. Equivalently, it means
that there are no line segments in $\partial \O$.

Let $\O \subseteq \C^d$, and let $\lambda, \mu$ be two distinct points in $\O$. 
Following \cite{aly16}, we shall call the pair $\delta = (\lambda, \mu)$ a {\em datum}.
A {\em Kobayashi extremal} for $\d$ is a holomorphic map $k: \D \to \O$ such that
both $\lambda$ and $\mu$ are in the range of $k$, and the pseudo-hyperbolic distance $\rho$ between the
pre-images of $\lambda$ and $\mu$ is minimized over all holomorphic maps from $\D$ to $\O$.
A {\em Carath\'eodory extremal} for $\d$ is a holomorphic map $\phi: \O \to \D$ that maximizes
$\rho(\phi(\lambda), \phi(\mu))$.

If $\d$ is a datum, we shall say that the Kobayashi extremal is {\em essentially unique} if,
given any two Kobayashi extremals $k_1$ and $k_2$, they are related by
$k_2 = k_1 \circ m$, where $m$ is a M\"obius automorphism of $\D$.
We shall call the range of a Kobayashi extremal with datum $\delta$ a {\em geodesic} through
$\d$.
If a set $V \subseteq \O$ has the property that for any $\lambda, \mu$ in $V$, a geodesic through
$(\lambda,\mu)$ is contained in $V$, we shall say that $V$ is {\em totally geodesic}.

 A theorem of L. Lempert \cite{lem81} asserts that if $\O$ is convex,
every Kobayashi extremal $k$  has a left inverse, \ie a Carath\'eodory extremal $\phi$ satisfying
\[
\phi(k(z)) \= z \quad \forall \ z \in \ \D .
\]
A consequence is that every geodesic is a retract, since $r = k \circ \phi$ is the identity on ${\rm Ran}(r)$.

Much of the difficulty in characterizing subsets of $\D^2$ with the extension property
in \cite{agmc_vn} stemmed from the fact that on the bidisk, Kobayashi extremals need not
be essentially unique. In strictly convex domains, 
Kobayashi extremals are essentially unique \cite[Prop 8.3.3]{jp93}. Kobayashi extremals are also essentially unique
in the symmetrized bidisk $G$ \cite{ay06}.

We shall need the following result, the Royden-Wong theorem. A complete proof is in 
Lemmata 8.2.2 and 8.2.4 and Remark 8.2.3 in \cite{jp93}. For a function on the unit disk that is
 in a Hardy space, we shall use the same symbol for the function on $\D$  and for its non-tangential limit
function on the circle $\T$. We shall use $\bullet$ to mean the bilinear form on $\C^d$
\[
z \bullet w \= \sum_{j=1}^d z_j w_j .
\]

\bt
\label{thmb1}
Let $\O$ be a bounded convex domain in $\C^d$, and $k : \D \to \O$ a Kobayashi extremal for
some datum. Then:

(i) $k(z) \in \partial \O$ for a.e. $z \in \T$.

(ii) There exists a non-zero function $h \in H^1(\D, \C^d)$ such that
\be
\label{eqb11}
\Re [ (\lambda - k(z) ) \bullet (\bar z h(z) )] < 0 \qquad \forall\  \lambda \in \O, \ {\rm and\ a.e.\ } z \in \T .
\ee

(iii) There exists a holomorphic $\phi: \O \to \D$ that satisfies $\phi \circ k$ is the identity on $\D$,
and satisfies the equation
\[
[ \lambda - k( \phi(\lambda)) ] \bullet h(\phi(\lambda)) \= 0 \quad \forall\ \lambda \in \O.
\]
\et

\begin{defin}
Let $\O$ be an open set in $\C^d$, and $V \subseteq \O$. We say that $V$ is relatively polynomially convex
if $ \overline{V} \cap \O = V$, and $\overline{V}$ is polynomially convex in $\C^d$.
\end{defin}

\begin{proposition}
\label{propb1}
Let $\O$ be a strictly convex bounded domain in $\C^d$.
Let $V$ be relatively polynomially convex in $\O$.
If $V$ has the extension property, then $V$ is totally geodesic.
\end{proposition}
\bp
Let us assume that $\lambda,\mu$ are distinct points in $V$, let $\G$ be the unique geodesic
through them, and let $k: \D \to \G \subset \O$ be a Kobayashi extremal. 
Assume that $\G $ is not contained in $\overline V$;
we shall derive a contradiction.

Since $\overline V$ is polynomially convex, there must exist some part of the boundary
of $\G$ that is not in $\overline V$. So there exists $\xi \in \partial \G,\ \eta \in \T$ and $\vare, \vare' > 0$,
so that $B(\xi, \vare) \cap \overline V = \emptyset$ and $k(\D(\eta, \vare') \cap \D) \subseteq B(\xi,\vare)$.

Let $h$ be the function from Theorem~\ref{thmb1}. Wiggling $\eta$ a little if necessary, we can assume that
both $h$ and $k$ have non-tangential limits at
$\eta$, and that \eqref{eqb11} holds for $z = \eta$. Then by part (ii) of Theorem~\ref{thmb1}, we have that
\[
\{ \Re[ (\lambda - \xi) \bullet ( \bar \eta h(\eta))] = 0 \}
\]
is a supporting plane for $\O$ that contains $\xi$. Since $\O$ is strictly convex,
small perturbations of this plane will only intersect
$\O$ in $B(\xi, \vare)$.
So there is a small triangle $\Gamma$ in $\D$, with a vertex at $\eta$, such that for
$z \in \Gamma$, we have
\[
\{ \lambda \in \O : \Re[(\lambda - k(z) ) \bullet h(z)] = 0 \} \cap V \= \emptyset .
\]
Therefore if $\lambda \in V$, we have $\phi(\lambda) \notin \Gamma$.

Now we use an idea of P. Thomas, \cite{th03}.
Let $g$ be the Riemann map from $\D \setminus \Gamma$ to $\D$.
Let $\psi = g \circ \phi$.
Then
\[
\rho (\psi(\lambda), \psi(\mu) ) \ > \ \rho(\phi(\lambda), \phi(\mu)) .
\]
Since $\O$ is convex, the function $\psi$ can be uniformly approximated
by polynomials on compact subsets,
so one can find a polynomial $p$ that maps $\O$ to $\D$ and 
satisfies
\[
\rho (p(\lambda), p(\mu) ) \ > \ \rho(\phi(\lambda), \phi(\mu)) .
\]
Since $V$ has the extension property, there is a function $F: \O \to \D$ that agrees with $p$ on $V$, 
and in particular
\[
\rho (F(\lambda), F(\mu) ) \ > \ \rho(\phi(\lambda), \phi(\mu)) .
\]
This contradicts the assumption that $\phi$ is a Carath\'eodory extremal for $(\lambda,\mu)$ in $\O$,
which it must be since it is a left inverse to $k$.
\ep

We can now prove Theorem~\ref{thma3}.

{\sc Proof of Theorem~\ref{thma3}}: The conclusion is obvious if $V$ is  a singleton, 
so we shall assume it contains at least two points. By Proposition~\ref{propb1}, we know that $V$ is totally geodesic.
We shall prove that in fact $V$ must either  be a single geodesic, and hence a one dimensional retract, or all of $\O$.

Let $\G$ be a geodesic in $V$, and assume that there is some point $a \in V \setminus \G$.
For each point $\lambda \in \G$, let $k_\lambda$ be the Kobayashi extremal that passes through
$a$ and $\lambda$, normalized by $k_\lambda(0) = a$ and $k_\lambda(r) = \lambda$ for some $0 < r < 1$.
Let $\DDD$ be a subdisk of $\G$ with compact closure.
Note that 
\[
\{ r : k_\lambda(r) = \lambda,\ \lambda \in \overline{\DDD} \}
\]
is bounded away from $0$ and $1$, since $r$ is the Kobayashi distance
between $\lambda$ and $a$.
Let $\BB = \D(\frac 12, \frac 14)$. Define
\beq
\Psi : \BB \times \DDD &\to & \O \\
(z, \lambda) &\mapsto & k_\lambda(z) .
\eeq
Claim: $\Psi$ is continuous and injective.

Proof of claim: Since geodesics are unique, any two geodesics that share two points must coincide. As $k_\lambda(0) = a$ for each $\lambda$, it follows that if $\lambda_1 \neq \lambda_2$,
then $k_{\lambda_1} (z_1) \neq k_{\lambda_2}(z_2)$ unless $z_1 = 0 = z_2$. Moreover, since each $k_\lambda$ has a left inverse, each $k_\lambda$ is injective, so $\Psi $ is injective.

To see that $\Psi$ is continuous, suppose that $\lambda_n \to \lambda$ and $z_n \to z$.
By Montel's theorem, every subsequence of $k_{\lambda_n}$ has a subsequence that
converges uniformly on compact subsets of $\D$ to a function $k : \D \to \O$.
But clearly $k(0) = a$ and $k$ maps some positive real number $r$ to $\lambda$,
where $r = \lim \rho(\lambda_n, a)$. Therefore $k = k_\lambda$.
So $\lim \Psi (\lambda_n, z_n)  = \Psi (\lambda, z)$, and hence $\Psi$ is continuous.
\oec

So $\Psi$ is a continuous injective map between two open subsets of $\C^2$;
so by the invariance of domain theorem, $\Psi$ is open. Therefore the range of
$\Psi$ is an open subset $U$ of $\O$, and since $V$ is totally geodesic, $U \subseteq V$. Since any point in $\O$ is connected by a geodesic to a point in $U$, and that 
geodesic must intersect $U$ in a continuum of points (in particular, more than one),
it follows that $V$ is all of $\O$.
\ep

\section{The ball}
\label{secc}

Let $\bad$ be the unit ball in $\C^d$, the set $\{ z \in \C^d: \sum_{j=1}^d |z_j|^2 < 1 \}$.
\bt
\label{thmc1} Let $V$ be be a relatively polynomially convex subset of $\bad$ that has the extension property.
Then $V$ is a retract of $\bad$.
\et
\bp
The result is obvious if $V$ is a singleton, so let us assume it has more than one point.
Composing with an automorphism of $\bad$, we can assume that $0 \in V$.
We will show that then $V$ is the image under a unitary map of $\bak$, for some $1 \leq k \leq d$.
First observe that if $a \in V \setminus \{ 0\}$, we can compose with a unitary so that it has
the form $(a_1, 0, \dots, 0)$. By Proposition~\ref{propb1}, we have that $\B_1 \subseteq V$.
Now we proceed by induction. Suppose that ${\B}_{k-1} \subsetneq V$, and that $b \in V \setminus 
{\B}_{k-1}$. Composing with a unitary, we can assume that 
$b = (b_1, \dots, b_k, 0, \dots, 0 )$ and $b_k \neq 0$.
Let $c$ be any point in $\bak$ with $\| c \| < |b_k|/2$.
Then
\[
c \= \frac{c_k}{b_k} (b_1, \dots, b_k, 0, \dots , 0) \ + \
(c_1 - \frac{c_k}{b_k} b_1, c_2 - \frac{c_k}{b_k} b_2, \dots, 0, \dots , 0).
\]
Then the first point on the right-hand side is in the geodesic connecting $b$ and $0$, so in in $V$;
and the second point is in ${\B}_{k-1}$, and so is also in $V$. Therefore the geodesic containing these
two points, which is 
the intersection of the plane containing these two points with $\bad$, is also in $V$, and hence
$c \in V$.

Continuing until we exhaust $V$, we conclude that $V = \Phi(\bak)$, for some $1 \leq k \leq d$ and
some automorphism $\Phi$ of $\bad$. \ep

\section{Strongly linearly convex domains}
\label{secd}

A domain $\O \subseteq \C^d$ is called {\em linearly convex} if, for every point $a \in \C^d \setminus \O$,
there is a complex hyperplane that contains $a$ and is disjoint from $\O$.
Now assume that $\O$ is 
 given by a $C^2$ defining function $r$ (\ie $\O = \{ z : r(z) < 0 \} $).
Then $\O$ is called {\em strongly linearly convex} if
\beq
\sum_{j,k=1}^d  \frac{\partial^2 r}{\partial z_j \partial \bar z_k} (a)
X_j \bar X_k
&\ > \ &
| \sum_{j,k=1}^d  \frac{\partial^2 r}{\partial z_j \partial \bar z_k} (a)
X_j  X_k | ,\\
&&\forall\ a \in \partial \O, \ X \in (\C^d)_* \ \text{with\ } 
\sum_{j}^d  \frac{\partial r}{\partial z_j } (a) X_j = 0 .
\eeq
Notice that we only check the inequality on complex tangent vectors.
Roughly speaking, a domain is strongly linearly convex if it is smooth, linearly convex and it remains linearly convex after small deformations.

A smooth domain $D$ is strictly convex if for any $a$ in the boundary of $D$ its defining function restricted to the real  tangent plane to $\partial D$ at $a$ 
attains a strict minimum at $a$, 
and it is strictly linearly convex if for any $a$ in the boundary of $D$ its defining function restricted to the
complex tangent plane (i.e. the biggest complex plane contained in the tangent plane)  to $\partial D$ at $a$ 
attains a strict minimum at $a$.
We shall call a smooth domain strongly convex if the Hessian of its defining function is strictly positive on the real tangent plane, and 
we shall call it strongly linearly convex if the Hessian of its defining function is strictly positive on the complex tangent plane. 
It is obvious that strong linear convexity implies strict linear convexity
 and that any strongly convex domain is strongly linearly convex. 

For  convenience  we shall  only consider  $\mathcal C^3$  domains, though the regularity actually needed
is
 $\mathcal C^{2,\epsilon}$ (which means that second order derivatives of the defining function are $\epsilon$-H\"older continuous).

If $\O$ is strongly linearly convex and has smooth boundary, it was proved by
Lempert \cite{lem84} that the Kobayashi extremals are unique and depend smoothly on points, vectors and even domains (in the sense of their defining functions)

\begin{lemma}[\cite{lem84}, Proposition 11, see also \cite{KZ16}, Proof of Theorem~3.1]
\label{lemd1}
Let $\Omega$ be strongly linearly convex with  $\mathcal C^3$ boundary, and 
let $f:\mathbb D\to \Omega$ be a complex geodesic.
Then there exists a domain
 $G\subset \mathbb D\times \mathbb C^{n-1}$ and a biholomorphic mapping $\Gamma:\bar \Omega \to \bar G$ such that $\Gamma(f(\lambda))=(\lambda, 0)$. Moreover, $\overline{ G} \cap (\mathbb T\times \mathbb C^{n-1}) = \mathbb T\times \{0\}$.
\end{lemma}

{\sc Proof of Theorem~\ref{thma4}:}

The proof is split  into two parts. The first one says that $V$ is totally geodesic. To get it we shall use the argument from the proof of Proposition \ref{propb1}, but using Lemma~\ref{lemd1} in lieu of the Royden-Wong theorem.

The second part says that any totally geodesic 
variety in $\Omega$ having an extension property is regular.

\begin{lemma}
\label{lem42}
	Let $\Omega$ and $V$ be as in Theorem~\ref{thma4}. Then $V$ is totally geodesic.
\end{lemma}

\begin{proof}
It was proven by Lempert in \cite{lem84} for smoothly bounded strongly linearly convex domains,
and in \cite{kw13} for ${\mathcal C}^2$-boundaries, that the Carath\'eodory and Kobayashi metrics coincide on $\O$, and that geodesics $f: \D \to \O$
 are essentially unique and $\mathcal C^{1/2}$-smooth on $\mathbb D$, so they extend continuously
to the closed unit disk.

 Take $z_1, z_2\in V$ and a complex geodesic $f$ passing through these points. 
Let $D = f(\D)$, and let $\Gamma$ be as in Lemma~\ref{lemd1}. Then $F(z)=z_1$ is a left inverse to $\Gamma(f(\lambda))=(\lambda, 0)$,  and $F\circ  \Gamma$ is a left inverse of $f$ that extends to be continuous on $\overline{\O}$.

Therefore $F \circ \Gamma(\overline{V})$ must contain the whole circle $\T$, by the argument used in the proof
of Proposition~\ref{propb1}.
By Lemma~\ref{lemd1}  the only points of $\overline{\O}$ on which $F \circ \Gamma$ is unimodular are on the boundary of $D$.
Therefore, since $V$ is relatively polynomially convex, we have $D \subseteq V$, as required.
\end{proof}

\begin{lemma}
\label{lem43}
Let $\Omega$ be a bounded open set in $\C^n$, and let $V \subseteq \Omega$ be a relatively polynomially convex set that
has the extension property. Then $V$ is a holomorphic subvariety of $\Omega$.
\end{lemma}
\begin{proof}
Let $b$ be any point in $\Omega\setminus V$.
Then there is a polynomial $p$ such that $|p(b)| > \| p \|_V$.
By the extension property, there is a function $\phi \in H^\i(\O)$ such that
$\phi |_V = p$ and $\| \phi \|_\O = \| p \|_V$. Let $\psi_b = \phi-p$.
Then $\psi_b$ vanishes on $V$ and is non-zero on $b$.
Therefore $V = \cap _{b \in \Omega\setminus V} Z_{\psi_b}$.

Locally, at any point $a$ in $V$, the ring of germs of holomorphic functions is Noetherian
\cite[Thm. B.10]{gun2}. Therefore $V$ is locally the intersection of finitely many zero zets of functions
in $H^\i(\O)$, and therefore is a holomorphic subvariety.
\end{proof}

In the next proof, we shall write $X_*$ to mean $X \setminus \{ 0 \}$.
\begin{lemma}
\label{lem44}
	Let $V$ be a
 totally geodesic
 holomorphic subvariety of a 
 strongly
 linearly convex domain $\Omega$,
and assume that $V$ has the extension property. Than $V$ is a complex manifold.
\end{lemma}

\begin{proof}
	
	Set $m=\dim_\C V$ and take $a\in V$. We want to show
 that $V$ is a complex submanifold near $a$. 
If $m =0$, there is nothing to prove, so we shall assume $m > 0$.
Then there is a complex geodesic contained in $V$ passing through $a$. Of course, it is a $1$-dimensional complex submanifold, so trivially a $2$-dimensional real submanifold.

{\bf Claim.} (i)  If $\mathcal F$ is a local $2k$ dimensional real submanifold 
	smoothly embedded in
 $V$ in a neighbourhood of $a\in \mathcal F$, and there is a geodesic $f:\mathbb D\to V$ such that $f(0)=a$ and $f'(0)\notin \omega T_a\mathcal F$ for any $\omega\in \mathbb T$, then there is a local real submanifold of dimension $2k+2$ that is contained in $V$ and that contains $a$.
 
(ii)	Moreover,  one can always find such a geodesic $f$ whenever $k<m$ or if $k=m$ and $V$ does not coincide with $\mathcal F$ near $a$. 
	
	

	To prove  part (i) of the claim,
	 let us take a geodesic $f$ such that $f(0)=a$, 
and $f'(0)\notin \omega T_a\mathcal F$ for any $\omega\in \mathbb T$. 
Take any $t_0\in (0,1)$ and set $a_0=f(t_0)$. There are a neighbourhood $U$ of $a$ and a smooth mapping $\Phi:U\times \mathbb D \to \mathbb C^d$ such that,  for any $z$ near $a$,
 a disc $\Phi(z,\cdot)$ is a geodesic passing through $a_0$ and $z$ such that $\Phi(z,0)=a_0$ and $\Phi(z, t_z)=z$ for some $t_z>0$. It also follows from Lempert's theorem
\cite{lem81,lem84} (see \cite{kw13} for an exposition)
 that $z\mapsto t_z$ is smooth. Observe that $\Phi(a,\lambda) = f(m_{t_0}(\lambda))$.
	Note that the mapping:
	$$\mathcal F\times \mathbb D\ni (z,\lambda)\mapsto \Phi(z, \lambda)$$
	sends $(a,t_0)$ to $a$. We shall show that its Jacobian is non-degenerate in a neighborhood of $a$. This will imply that the image of this mapping is a smooth $2k+2$ dimensional real submanifold, thus proving the claim.
	
	Let $p:(-1,1)^{2k} \to \mathcal F$ give local coordinates for $\mathcal F$, $p(0)=a$. We need to compute the Jacobian of $(s,\lambda)\mapsto \Psi(s,\lambda):=\Phi(p(s), \lambda)$ at $(0,t_0)$. Write $\lambda$ in coordinates $(x,y)\in \mathbb R^2$ and $\Psi = (\Psi_1,\ldots, \Psi_{2n})$. 


The Jacobian matrix of $\Psi$ is the $(2k+2)$-by-$2n$ matrix with columns
\[
\begin{pmatrix}
\frac{\partial \Psi_i}{\partial s_j}
\\ \\
\frac{\partial \Psi_i}{\partial x}
\\ \\
\frac{\partial \Psi_i}{\partial y}
\end{pmatrix}
\]
Differentiating $\Psi(s, r(s)) = p(s)$, where $r(s) = t_{p(s)},$ we get
\[
\frac{\partial \Psi_i}{\partial s_j}
+
\frac{\partial \Psi_i}{\partial x} \frac{\partial r}{\partial s_j} 
+
\frac{\partial \Psi_i}{\partial y} \frac{\partial r}{\partial s_j} 
\ = \
\frac{\partial p_i}{\partial s_j} 
.
\]
So the rank of the Jacobian of $\Psi$ is the same as the rank of the matrix with columns
\be
\label{eqppsi}
\begin{pmatrix}
\frac{\partial p_i}{\partial s_j}
\\ \\
\frac{\partial \Psi_i}{\partial x}
\\ \\
\frac{\partial \Psi_i}{\partial y}
\end{pmatrix}
\ee
Since $\Psi(0, \lambda) = f(m_{t_0}(\lambda))$,  by differentiating with respect to $\lambda$
we find that $\partial \Psi/\partial \lambda(0, t_0) = - f'(0)$. 
So if  $f'(0)\not\in \omega T_a\mathcal F$ for any unimodular $\omega$, the rank of \eqref{eqppsi} is
2 more than the rank of $(\frac{\partial p_i}{\partial s_j})$, which is $2k$.
Therefore the Jacobian of $\Psi$ is of rank $2k+2$, and so we have established (i) of the claim.	
	
\vs
	
	Proof of part (ii), the existence of the geodesic.

Case: when $k < m$.
 For $z \in V\setminus \{a \}$, let
 $f_z$ denote a complex geodesic such that $f_z(0)=a$ and $f_z(t_z)=z$
	for some $t_z>0$. Note that the real dimension of the set $\{tf_z'(0):\ z\in V, t>0\}$ is equal at least to $2m$. 
	To see this one can proceed as follows: take $z_0\in V_{reg} \setminus \{ a \}$.
 Let $W$ be an $m-1$-dimensional complex surface near $z_0$ that is contained in $V$ and that it transversal to $f_{z_0}'(t_{z_0})$. Then the mapping $\mathbb D_*\times W\to \mathbb C^m$ given by $(\lambda, z) \mapsto \lambda f_z'(t_{z_0})$ is locally injective for $z$ close to $z_0$, so its image is $2m$ real dimensional.
	
	On the other hand, the real dimension of  the set $$
\mathbb T\cdot T_a\mathcal F \ :=\ \{\omega X:\ \omega\in \mathbb T, X\in T_a \mathcal F\}
$$ is equal to $2k+1$ or $2k$. 
So the existence of the geodesic follows.

Case: $k=m$. Then $\mathcal F$ is a real submanifold of dimension $2m$ that is contained in an analytic set of complex dimension $m$.
The singular points of $V$ are a subset of complex dimension at most $m-1$.
 Looking at the regular points of $V$, 
 we  get that $\mathcal F$ is a totally complex manifold,
  which means that its tangent space has a complex structure, on the regular points.
Since $\mathcal F$ is a submanifold near $z$, its tangent space depends continuously on $z$,
so $\mathcal F$ has a complex structure on its tangent space at all points.
So by Lemma I.7.15 in \cite{dem11}, it follows
that $\mathcal F$ must be a complex submanifold.

 Change variables so that $\mathcal F$ is the graph of a holomorphic mapping $\{(z', h(z')): ||z'-a'||<2\epsilon\}$ near $a=(a',a'')$. Denote by $S_\epsilon$ the set $$
S_\epsilon \= \{(z',h(z')):\ ||z'-a'||=\epsilon\} .
$$
 If there is $z'\in S_\epsilon$ such that $f_{(z', h(z'))}'(0)\notin \omega T_a \mathcal F = T_a \mathcal F$ 
we have shown (ii).
 Otherwise there exists $\epsilon_0 > 0$ such that,  for any $0 < \epsilon < \epsilon_0$,
  the mapping
  \se\att
\begin{eqnarray}
\nonumber
\alpha_\epsilon: (0, \infty)  \times S_\epsilon &\ \to \ & (T_a \mathcal F)_* \\
\label{eqd1}
(\stretchit,z) & \mapsto &  \stretchit f_z'(0)
\end{eqnarray}
 is injective,
 since geodesics are unique (with our normalizations,
namely
that $0$ maps to $a$ and $f_z^{-1}(z)$ is positive).
Therefore by the invariance of domain theorem, it is surjective.

%
%
	
	Thus, if in an arbitrarily small neighbourhood of $a$ there is a point $w\in V\setminus \mathcal F$, then a complex geodesic for $a$ and $w$ satisfies the desired condition. To see this take a geodesic $g$ such that $g(0)=a$, $g(t_w)=w$, $t_w>0$, and suppose that $g'(0)\in T_a\mathcal F$. 
(Note that we have shown that $\mathcal F$ is a complex manifold	, so its tangent space is invariant under
complex multiplication).
	
 By the surjectivity of $\alpha_\epsilon$ that we have just shown,
  we get that there is $(\stretchit,z)\in (0,\infty) \times S_\epsilon$ such that $\stretchit f_z'(0) = g'(0)$. The uniqueness of geodesics implies that $f_z=g$ and consequently $g(t_z) = f_z(t_z) =z$. 

Let \[
H(z',z'') \ := \ (||z' - a'||^2,||z''-h(z')||^2) .
\]
Then $H\circ g$ is a real analytic mapping. Observe that $H\circ g (t_z) = (\epsilon^2, 0)$.
Now $t_z$ depends on $z$ which depends on $\epsilon$.
So for every $0 < \epsilon < \epsilon_0$, there is a $t_z(\epsilon)$ in $(0,1)$ so that
 $H\circ g (t_z (\epsilon) ) = (\epsilon^2, 0)$.
 By analyticity, it follows that the second component of $H$ is identically zero in $(0,t_w)$, so $w=g(t_w)$
 is in $\mathcal F$, a contradiction.

%

\bs

Having proved the claim, we finish the proof of the Lemma in the following way.
Let $\F_1$ be a complex geodesic contained in $V$ that passes through $a$. If $m  > 1$,
for $1 \leq k \leq m-1$,  by part (ii) of the claim we can find a geodesic $f$ through $a$ such that $f'(0)$ is not in 
$\omega T_a \F_k$ for any complex $\omega$, and so by part (i) we can find a real $2k$-dimensional
smoothly embedded manifold $\F_{2k}$ contained in $V$ and containing $a$.
If $\F_m$ does not equal $V$ near $a$, then by part (ii) we could find a $\F_{m+1}$ contained in $V$,
which is ruled out by the dimension count.
Therefore $\F_m =V$ near $a$, so $V$ is a smooth real submanifold, and hence, as already shown, a
complex submanifold.
\end{proof}

Combining Lemmas~\ref{lem42} and \ref{lem44}, we finish the proof of Theorem~\ref{thma4}.

{\sc Proof of Corollary \ref{cora1}:}
 If $n=2$, then we have shown that if $V$ has the extension property and is not a singleton
or all of $\Omega$, then it is a one dimensional totally geodesic set.
So there exists a Kobayashi extremal $f: \D \to \O$ whose range is exactly $V$.
By Lempert's theorem, there is a left inverse $L : \O \to \D$.
Let $r = f \circ L$; this is a retract from $\O$ onto $V$.
\ep

\section{Spectral sets}
\label{sece}

Let $\Omega$ be a bounded open set in $\C^d$, and $V \subseteq \O$;
they shall remain  fixed for the remainder of this section.
Let $A(\O)$ denote the algebra of holomorphic functions on $\O$ that extend to be continuous on the closure
$\overline{\O}$, equipped with the supremum norm. For any positive finite measure $\mu$ supported on 
$\overline{\O}$, let $\atmu$ denote the closure of $A(\O)$ in $L^2(\mu)$.

A point $\lam \in \O$ is called a bounded point evaluation of $\atmu$ if there exists a constant $C$ so that
\be
\label{eqe1}
| f(\lam) | \ \leq \| f \| \quad \forall f \in A(\O) .
\ee
If \eqref{eqe1} holds, then by the Riesz representation theorem there is a function $k^\mu_\lam \in \atmu$ such that
\be
\label{eqe2}
 f(\lam)  \ = \ \la  f, k^\mu_\lam \ra \quad \forall f \in A(\O) .
\ee
Given a set $\Lambda \subseteq \O$, we say the measure $\mu$ is dominating for $\Lambda$ if every
point of $\Lambda$ is a bounded point evaluation for $\atmu$.
We shall need the following theorem of Cole, Lewis and Wermer \cite{clw92}; similar results
were proved by Amar \cite{am77} and Nakazi \cite{nak90}. See \cite[Thm. 13.36]{ampi} for an exposition.
For the polydisk or the ball, one can impose extra restrictions on the measures $\mu$ that need to be checked
\cite{am03,tw09}.

\bt
\label{thme1}
Let $\{ \lambda_1, \dots , \lambda_N \} \subseteq \O$ and $\{ w_1, \dots, w_N \} \subseteq \C$ be given.
For every $ \epsilon > 0$, there exists a function $f \in A(\O)$ of norm at most $1+\vare$ that
satisfies
\[
f(\lam_i) \= w_i, \quad \forall\ 1 \leq i \leq N 
\]
if and only if, for every measure $\mu$ supported on $\partial \O$ that dominates
$\{ \lambda_1, \dots , \lambda_N \}$, we have
\be
\label{eqe21}
\left[ (1 - w_i \overline{w_j}) \la k^\mu_{\lam_j}, k^\mu_{\lam_i} \ra_{\atmu} \right]_{i,j=1}^N \ \geq \ 0 .
\ee
\et

{\sc Proof of Theorem~\ref{thmab2}:}
We shall actually show slightly more: for any function $f \in \A$, we shall show that
$f$ can be extended to a function $\phi$ in $H^\i(\O)$ of the same norm that agrees with $f$ on $V$ if and 
only if
\be
\label{eqe3}
\| f(T) \| \ \leq \ \sup_{\lambda \in V} |f(\lambda)| \quad \forall \ T\  {\rm subordinate\ to\ } V.
\ee

One direction is easy. Suppose that $V$ has the $\A$ extension property, and $T$ is subordinate to $V$.
By the extension property, there exists $\phi \in H^\i(\O)$ such that $\phi |_V = f|_V$, and
$\| \phi \| = \sup_V |f|$.
Since $T$ is subordinate to $V$ and $f - \phi$ vanishes on $V$, we get that $f(T) = \phi(T)$, so
\[
\| f(T) \| \= \| \phi (T) \| \ \leq \ \| \phi \| \=   \sup_{\lambda \in V }|f (\lambda)| ,
\]
where the inequality comes from the fact that $\O$ is a spectral set for $T$.

The converse direction is more subtle. Suppose $\sup_{\lambda \in V }|f (\lambda)| =1$,
and assume we cannot extend $f$ to a function $\phi$ of norm one in $H^\i(\O)$. We shall construct $T$ 
subordinate to $V$ so that \eqref{eqe3} fails.

 Let $\{ \lambda_j \}$ be a countable dense set in $V$. Let $w_j = f(\lam_j)$. 
For each $N$, let $E_N = \{ \lambda_1, \dots, \lambda_N \}$.
If, for each $N$, one could find $\phi_N \in A(\O)$ of norm at most $1 + \frac{1}{N}$ and that satisfies
\[
\phi_N( \lam_i) \= f(\lam_i), \quad \forall\ 1 \leq i \leq N,
\]
then by Montel's theorem some subsequence of $(\phi_N)$ would converge to a function $\phi$
in the closed unit ball of $H^\i(\O)$ that agreed with $f$ on a dense subset of $V$, and hence on all of $V$.
So by Theorem~\ref{thme1}, there must be some $N$, and some measure $\mu$ that dominates $E_N$, so that
\eqref{eqe21} fails. Fix such an $N$ and such a $\mu$.

Let $k_j = k^\mu_{\lam_j}$, and let $\M$ be the linear span of 
$\{k_1, \dots, k_N \}$.
Define a $d$-tuple of operators $T = (T_1, \dots , T_d)$ on $\M$
\be
\label{eqe5}
T_r^*\  k_j \= \overline{\lam_j^r}\  k_j ,\quad 1 \leq r \leq d,\ 1\leq j \leq N .
\ee
(We write $\lam_j^r$ for the $r^{\rm th}$ component of $\lam_j$).  
By \eqref{eqe5}, the spectrum of $T^*$ is $\{ \overline{\lam_1}, \dots, \overline{\lam_N} \}$, so the spectrum of
$T$ is $E_N \subseteq V$.

If $\psi \in H^\i(\O)$, let $\psi^\cup(z) = \overline{\psi(\overline{z})}$.
From \eqref{eqe5} it follows that
\be
\label{eqe6}
\psi(T)^* k_j \= \psi^\cup(T^*) k_j \= \overline{\psi(\lam_j)} k_j .
\ee
Indeed, since $T$ is a $d$-tuple of matrices, $\psi(T)$ only depends on the value of $\psi$ and a finite number of
derivatives on the spectrum of $T$; so one can approximate $\psi$ by a polynomial, and for polynomials
\eqref{eqe6} is immediate.

We have that $T$ is the compression to $\M$ of multiplication by the coordinate functions on $\atmu$.
Let $P$ be orthogonal projection from $\atmu$ onto $\M$. If $g \in A(\O)$, then 
\[
\| g(T) \| \= \| P M_g P \| \ \leq \ \|g \|_{A(\O)},
\]
where $M_g$ is mutliplication by $g$ in $\atmu$.
So $T$ has $\O$ as a spectral set. By \eqref{eqe6}, if $\psi$ vanishes on $V$, then $\psi(T) = 0$,
so $T$ is subordinate to $V$.

We want to show that $\| f(T) \| > 1$. If it were not, then 
\be
\label{eqe7}
I - f(T) f(T)^* \ \geq \ 0 .
\ee
Evaluating the left-hand side of \eqref{eqe7} on $v = \sum a_j k_j$ and using \eqref{eqe6}, one gets
\be
\label{eqe8}
\la (I - f(T) f(T)^*) v, v \ra \= \sum_{i,j=1}^N \overline{a_i} a_j (1 - w_i \overline{w_j}) \la k_j, k_i \ra
\ee
But we chose $N$ and $\mu$ so that for some choice of $a_j$, \eqref{eqe8} is negative.
This contradicts \eqref{eqe7}.
\ep

\bibliographystyle{amsplain} 

\bibliography{references}

\end{document}